\documentclass[a4,psfig,english,11pt,epsf,portrait]{article}
\usepackage{amsmath,amsfonts,latexsym,amscd,amssymb}
\usepackage{epsfig}
\usepackage{amsmath}
\usepackage{amsthm}
\usepackage{amssymb}
\usepackage{enumerate}
\usepackage{color}
\usepackage{hyperref}
\hypersetup{hidelinks}
\definecolor{purple}{rgb}{0.65, 0, 1}
\definecolor{orange}{rgb}{1,.5,0}
\definecolor{brown}{rgb}{.9,.73,.26}

\newtheorem{theorem}{Theorem}[section]
\newtheorem{remark}{Remark}[section]

\newtheorem{definition}{Definition}[section]

\newtheorem{corollary}[theorem]{Corollary}
\numberwithin{equation}{section}

\usepackage{color}

\title{Fujita type results for a parabolic inequality with a non-linear convolution term on the Heisenberg group}

\author{
Ahmad Z. Fino, Mokhtar Kirane, Bilal Barakeh, Sebti Kerbal}

\date{}

\begin{document}
	\maketitle

\begin{abstract}
The purpose of this paper is to investigate the non-existence of global weak solutions of the following degenerate inequality on the Heisenberg group
$$
\begin{cases}
u_{t}-\Delta_{\mathbb{H}}u\geq (\mathcal{K}\ast_{_{\mathbb{H}}}|u|^p)|u|^q ,\qquad {\eta\in \mathbb{H}^n,\,\,\,t>0,}
 \\{}\\ u(\eta,0)=u_{0}(\eta), \qquad\qquad\qquad\quad \eta\in \mathbb{H}^n,
 \end{cases}
$$
where $n\geq1$, $p,q>0$, $u_0\in L^1_{\hbox{\tiny{loc}}}(\mathbb{H}^n)$, $\Delta_{\mathbb{H}}$ is the Heisenberg Laplacian, and $\mathcal{K}:(0,\infty)\rightarrow(0,\infty)$  is a continuous function satisfying $\mathcal{K}(|\cdotp|_{_{\mathbb{H}}})\in L^1_{\hbox{\tiny{loc}}}(\mathbb{H}^n)$ which decreases in a vicinity of infinity. In addition, $\ast_{_{\mathbb{H}}}$ denotes the convolution operation in $\mathbb{H}^n$. Our approach is based on the non-linear capacity method.
\end{abstract}

\medskip

\noindent {\bf MSC 2020 Classification}:  Primary 35A01; 35R03; 35B53; Secondary 35B33; 35B45

\noindent {\bf Keywords:} Parabolic inequalities, Heisenberg group, global non-existence, non-linear capacity estimates


\section{Introduction}

In this paper, we are interested in the non-existence of global weak solutions of the following parabolic inequality on the Heisenberg group
\begin{equation}\label{1}
\begin{cases}
u_{t}-\Delta_{\mathbb{H}}u\geq (\mathcal{K}\ast_{_{\mathbb{H}}}|u|^p)|u|^q ,\qquad {\eta\in \mathbb{H}^n,\,\,\,t>0,}
 \\{}\\ u(\eta,0)=u_{0}(\eta), \qquad\qquad\qquad\quad \eta\in \mathbb{H}^n,
 \end{cases}
\end{equation}
where $n\geq1$, $p,q>0$, $u_0\in L^1_{\hbox{\tiny{loc}}}(\mathbb{H}^n)$, and $\Delta_{\mathbb{H}}$ is the Heisenberg Laplacian. The function $\mathcal{K}:(0,\infty)\rightarrow(0,\infty)$ is continuous, it satisfies $\mathcal{K}(|\eta|_{_{\mathbb{H}}})\in L^1_{\hbox{\tiny{loc}}}(\mathbb{H}^n)$ and there exists $R_0>1$ such that $\inf\limits_{r\in(0,R)}\mathcal{K}(r)=\mathcal{K}(R)$ for all $R>R_0$. The non-linear convolution term $\mathcal{K}\ast_{_{\mathbb{H}}}|u|^p$ is the Heisenberg convolution between $\mathcal{K}$ and $|u|^p$ defined by
$$(\mathcal{K}\ast_{_{\mathbb{H}}}|u|^p)(\eta)=\int_{\mathbb{H}^n}\mathcal{K}(d_{_{\mathbb{H}}}(\eta,\xi))|u(\xi)|^p\,d\xi=\int_{\mathbb{H}^n}\mathcal{K}(|\xi|_{_{\mathbb{H}}})|u(\eta\circ\xi^{-1})|^p\,d\xi,$$
where $d_{_{\mathbb{H}}}$ and $|\cdotp|_{_{\mathbb{H}}}$ are, respectively, the Kor\'anyi distance and Kor\'anyi norm defined below.
\begin{remark}
Typical examples of $\mathcal{K}$ are the constant functions as well as
$$\mathcal{K}(r)=r^{-\alpha},\quad\alpha\in(0,Q),$$
where $Q=2n+2$ is the homogeneous dimension of $\mathbb{H}^n$.
\end{remark}

The Heisenberg group  is the Lie group $\mathbb{H}^n=\mathbb{R}^{2n+1}$ equipped with the following law
$$\eta\circ\eta^\prime=(x+x^\prime,y+y^\prime,\tau+\tau^\prime+2(x\cdotp y^\prime-x^\prime\cdotp y)),$$
where $\eta=(x,y,\tau)$, $\eta^\prime=(x^\prime,y^\prime,\tau^\prime)$, and $\cdotp$ is the scalar product in $\mathbb{R}^n$. Let us denote the parabolic dilation in $\mathbb{R}^{2n+1}$ by $\delta_\lambda$, namely, $\delta_\lambda(\eta)=(\lambda x, \lambda y, \lambda^2 \tau)$ for any $\lambda>0$, $\eta=(x,y,\tau)\in\mathbb{H}^n$. The Jacobian determinant of $\delta_\lambda$ is $\lambda^Q$, where $Q=2n+2$ is the homogeneous dimension of $\mathbb{H}^n$.\\
The homogeneous Heisenberg norm (also called Kor\'anyi norm) is derived from an anisotropic dilation on the Heisenberg group and is defined by
$$|\eta|_{_{\mathbb{H}}}=\left(\left(\sum_{i=1}^n (x_i^2+y_i^2)\right)^2+\tau^2\right)^{\frac{1}{4}}=\left((|x|^2+|y|^2)^2+\tau^2\right)^{\frac{1}{4}},$$
where $|\cdotp|$ is the Euclidean norm associated to $\mathbb{R}^n$; the associated Kor\'anyi distance between two points $\eta$ and $\xi$ of $\mathbb{H}$ is defined by
$$d_{_{\mathbb{H}}}(\eta,\xi)=|\xi^{-1}\circ\eta|_{_{\mathbb{H}}},\quad \eta,\xi\in\mathbb{H},$$
where $\xi^{-1}$ denotes the inverse of $\xi$ with respect to the group action, i.e. $\xi^{-1}=-\xi$. \\
The left-invariant vector fields that span the Lie algebra are given by
$$X_i=\partial_{x_i}-2 y_i\partial_\tau,\qquad Y_i=\partial_{y_i}+2 x_i\partial_\tau.$$
The Heisenberg gradient is given by
\begin{equation}\label{48}
\nabla_{\mathbb{H}}=(X_1,\dots,X_n,Y_1,\dots,Y_n),
\end{equation}
and the sub-Laplacian (also called Kohn Laplacian) is defined by
\begin{equation}\label{40}
\Delta_{\mathbb{H}}=\sum_{i=1}^n(X_i^2+Y_i^2)=\Delta_x+\Delta_y+4(|x|^2+|y|^2)\partial_\tau^2+4\sum_{i=1}^{n}\left(x_i\partial_{y_i\tau}^2-y_i\partial_{x_i\tau}^2\right),
\end{equation}
where $\Delta_x=\nabla_x\cdotp\nabla_x$ and $\Delta_y=\nabla_y\cdotp\nabla_y$ stand for the Laplace operators on $\mathbb{R}^n$.

\subsection{Historical background}
\subsubsection{Results on $\mathbb{R}^n$}
Let us start by the pioneering paper of Fujita \cite{Fuj} where the following semi-linear heat equation 
\begin{equation}\label{heat}
\left\{
\begin{array}{ll}
\,\,\displaystyle{u_t-\Delta u=|u|^{p}}&\displaystyle {x\in {\mathbb{R}^n},\;t>0,}\\
\\
\displaystyle{u(x,0)=u_0(x)}&\displaystyle{x\in {\mathbb{R}^n},}
\end{array}
\right. \end{equation}
has been investigated. The exponent $p_F=1+2/n$ is known as the 
critical Fujita exponent of \eqref{heat}. Namely, for $p<p_F$, Fujita \cite{Fuj} proved the non-existence of non-negative global-in-time solution for any non-trivial initial condition, and for $p>p_F$, global solutions do exist for any sufficiently small non-negative initial data. The proof of blowing-up non-negative solutions in the critical case $p=p_F$ was completed in \cite{11, 17, 22}.

The study of convolution terms in the non-linearity of evolution equations has been used to model various quantities in gravity and quantum physics. For instance, the equation
\begin{equation}\label{convolution1}
i\psi_{t}-\Delta \psi=(|x|^{\alpha-n}\ast \psi^2)\psi ,\qquad x\in \mathbb{R}^n,\,\,\,t>0,
\end{equation}
$\alpha\in(0,n)$, $n\geq1$, was introduced by D.H. Hartee \cite{Hartee1,Hartee2,Hartee3} in 1928, shortly after the publication of the Schr\"{o}dinger equation, in order to study the non-relativistic atoms, using the concept of self-consistency. The stationary case of \eqref{convolution1} for $n=3$ and $\alpha=2$ is known is the literature as the Choquard equation and was introduced in \cite{Pekar} as a model in quantum theory.

To our knowledge, the first results in solving \eqref{heat} with a nonlocal non-linearity defined by the convolution operation i.e.
\begin{equation}\label{realnonlocal}
\begin{cases}
u_{t}-\Delta u\geq (\mathcal{K}\ast |u|^p)|u|^q ,\qquad {x\in \mathbb{R}^n,\,\,\,t>0,}
 \\{}\\ u(x,0)=u_{0}(x), \qquad\qquad\qquad x\in \mathbb{R}^n,
 \end{cases}
\end{equation}
were done recently in 2022 by Filippucci and Ghergu \cite{Filippucci1,Filippucci2}, where $n\geq1$, $p,q>0$, and $u_0\in L^1_{\hbox{\tiny{loc}}}(\mathbb{R}^n)$. The non-linear convolution term $\mathcal{K}\ast |u|^p$ is defined by
$$(\mathcal{K}\ast |u|^p)(x)=\int_{\mathbb{R}^n}\mathcal{K}(|x-y|)|u(y)|^p\,dy=\int_{\mathbb{R}^n}\mathcal{K}(|y|)|u(x-y)|^p\,dy.$$
Indeed, in \cite{Filippucci1,Filippucci2} the authors studied even more general cases than \eqref{realnonlocal} which includes quasilinear operators such as m-Laplacian. They proved that \eqref{realnonlocal} has no nontrivial 
global weak solutions whenever $p+q>2$, $u_0\in L^1(\mathbb{R}^n)$,
$$\int_{\mathbb{R}^n}u_0(x)\,dx>0\qquad\hbox{and}\qquad \limsup_{R\longrightarrow\infty}\mathcal{K}(R)R^{\frac{2n+2}{p+q}-n}>0.$$
They have illustrated their results by discussing the case of pure powers in the potential $\mathcal{K}(r)=r^{-\alpha},$ $\alpha\in(0,1),$ and proved that problem \eqref{realnonlocal} has no global weak solutions whenever $u_0\in L^1(\mathbb{R}^n)$, $2<p+q\leq\frac{2(n+1)}{n+\alpha},$ and $\displaystyle\int_{\mathbb{R}^n}u_0(\eta)\,d\eta>0$.

\subsubsection{Motivation} Motivated by the above facts, the aim of this paper is to generalize the non-existence of global solutions results to the case of Heisenberg group (Theorem \ref{theo1}-i below). In addition, we get optimal results under some restrictions on the initial data (Theorem \ref{theo1}-ii below).\\
Our proof approach is based on the method of non-linear capacity estimates specifically adapted to the nature of the Heisenberg groups. The non-linear capacity method was introduced to prove the non-existence of global solutions in $\mathbb{R}^n$ by  Baras and Pierre \cite{Baras-Pierre}, then used by Baras and Kersner in \cite{Baras}; later on, it was developed by Zhang in \cite{Zhang} and Pohozaev and Mitidieri in \cite{PM1}.  It was also used by Kirane and Guedda \cite{Kirane-Guedda}, then by Kirane et al.  in \cite{KLT},  and Fino et al in \cite{Fino1,Fino3,Fino4}.  While, in $\mathbb{H}^n$, it was used by Pohozaev and Veron in \cite{Pohoz1}, by D'Ambrosio in \cite{Ambrosio}, and by Jleli-Kirane-Samet in \cite{Jleli-Kirane} and recently by Fino-Ruzhansky-Torebek in \cite{Fino2}. 

\section{Main results}
Our main results are about the nonexistence of global weak solutions of \eqref{1} for various values of exponents $p$ and $q$.\\
Let us start with the definition of the weak solution of \eqref{1}.
\begin{definition}\textup{(Weak solution of \eqref{1})}${}$\\
Let $u_0\in L^1_{\hbox{\tiny{loc}}}(\mathbb{H}^n)$ and $T>0$. We say that $u\in L_{\hbox{\tiny{loc}}}^p((0,T)\times\mathbb{H}^n)$ is a weak solution of \eqref{1} on $[0,T)\times\mathbb{H}^n$ if
$$ (\mathcal{K}\ast_{_{\mathbb{H}}}|u|^p)|u|^q \in L_{\hbox{\tiny{loc}}}^1((0,T)\times\mathbb{H}^n),$$
and
\begin{equation}\label{weaksolution2}\begin{split}
&\int_{\mathbb{H}^n}u(\tau,\eta)\psi(\tau,\eta)\,d\eta-\int_{\mathbb{H}^n}u(0,\eta)\psi(0,\eta)\,d\eta\nonumber\\
&\geq\int_0^\tau\int_{\mathbb{H}^n}(\mathcal{K}\ast_{_{\mathbb{H}}}|u|^p)|u|^q\psi(t,\eta)\,d\eta\,dt+\int_0^\tau\int_{\mathbb{H}^n}u\,\Delta_{\mathbb{H}}\psi(t,\eta)\,d\eta\,dt\nonumber\\
&\quad+\int_0^\tau\int_{\mathbb{H}^n}u\,\psi_t(t,\eta)\,d\eta\,dt,
\end{split}\end{equation}
holds for all non-negative compactly supported test function $\psi\in C^{1,2}_{t,x}([0,T)\times\mathbb{H}^n)$, and $0\leq\tau<T$. If $T=\infty$,  $u$ is called a global  in time weak solution to \eqref{1}.
\end{definition}

\begin{theorem}\label{theo1}
Let $n\geq1,$ $p,q>0$ such that $p+q>2$.
\begin{itemize}
\item[(i)] If $u_0\in L^1(\mathbb{H}^n)$,
$$\int_{\mathbb{H}^n}u_0(\eta)\,d\eta>0\qquad\hbox{and}\qquad \limsup_{R\longrightarrow\infty}\mathcal{K}(R)R^{\frac{2Q+2}{p+q}-Q}>0,$$
then problem \eqref{1} has no global weak solutions.
\item[(ii)] If $ u_0\in L^1_{\hbox{\tiny{loc}}}(\mathbb{H}^n)$,
$$u_0(\eta)\geq\varepsilon(1+|\eta|_{_{\mathbb{H}}}^2)^{-\gamma/2}\,\,\,\,\hbox{and}\,\,\,\,\,
\liminf_{R\longrightarrow\infty}\left(\mathcal{K}(R)^{-1}\,R^{\gamma(p+q-1)-2-Q}\right)=0,$$
for some positive constant $\varepsilon>0$ and any exponent $\gamma>0$, then problem \eqref{1} has no global weak solutions.\\
\end{itemize}
\end{theorem}

\noindent Note that, when 
$$\mathcal{K}(r)=r^{-\alpha},\quad\alpha\in(0,Q),$$
i.e.
\begin{equation}\label{2}
\begin{cases}
u_{t}-\Delta_{\mathbb{H}}u\geq (|\eta|^{-\alpha}_{_{\mathbb{H}}}\ast_{_{\mathbb{H}}}|u|^p)|u|^q ,\qquad {\,\eta\in \mathbb{H}^n,\,\,\,t>0,}
 \\{}\\ u(\eta,0)=u_{0}(\eta), \qquad\qquad\qquad\qquad\quad \eta\in \mathbb{H}^n,
 \end{cases}
\end{equation}
and due to the fact that
\begin{eqnarray*}
\limsup_{R\longrightarrow\infty}\mathcal{K}(R)R^{\frac{2Q+2}{p+q}-Q}>0&\Longleftrightarrow&\lim_{R\longrightarrow\infty}R^{\frac{2Q+2}{p+q}-Q-\alpha}>0\\
&\Longleftrightarrow&\frac{2Q+2}{p+q}-Q-\alpha\geq0,
\end{eqnarray*}
and
\begin{eqnarray*}
\liminf_{R\longrightarrow\infty}\left(\mathcal{K}(R)^{-1}\,R^{\gamma(p+q-1)-2-Q}\right)=0&\Longleftrightarrow&\lim_{R\longrightarrow\infty}R^{\gamma(p+q-1)-2-Q+\alpha}=0\\
&\Longleftrightarrow&\gamma(p+q-1)-2-Q+\alpha<0,
\end{eqnarray*}
we have the following
\begin{corollary}\label{cor1}
Let $n\geq1,$ $p,q>0$, $\alpha\in(0,1)$.\\
\begin{itemize}
\item[(i)] Assume that $u_0\in L^1(\mathbb{H}^n)$ such that
$$\int_{\mathbb{H}^n}u_0(\eta)\,d\eta>0.$$
If
$$2<p+q\leq\frac{2(Q+1)}{Q+\alpha},$$
then problem \eqref{2} has no global weak solutions.
\item[(ii)] Assume that $u_0\in L^1_{\hbox{\tiny{loc}}}(\mathbb{H}^n)$, and there exists a constant $\varepsilon>0$ such that, for every $0<\gamma<Q+\alpha$, the initial datum verifies the following assumption:
$$u_0(\eta)\geq\varepsilon (1+|\eta|_{_{\mathbb{H}}}^2)^{-\gamma/2}.$$
If
$$2<p+q<\frac{2+Q+\gamma-\alpha}{\gamma},$$
then problem \eqref{2} has no global weak solutions.
\end{itemize}
\end{corollary}
\begin{remark}
Note that, we have chosen $\alpha\in(0,1)=(0,1)\cap(0,Q)$ to ensure that $2<\frac{2(Q+1)}{Q+\alpha}$.\\
In addition, due to $$\gamma<Q+\alpha\Longleftrightarrow \frac{2(Q+1)}{Q+\alpha}<\frac{2+Q+\gamma-\alpha}{\gamma},$$
the result in Corollary \ref{cor1}-(ii) positively affects on the nonexistence of global weak solutions in Corollary \ref{cor1}-(i). 
\end{remark}

\begin{proof}[Proof of Theorem \ref{theo1}] (i)
The proof is by contradiction. Suppose that $u$ is a global weak solution of (\ref{1}), then, for all $T\gg1$, we have
\begin{eqnarray*}
&{}&\int_{\mathbb{H}^n}u(T,\eta)\psi(T,\eta)\,d\eta-\int_{\mathbb{H}^n}u(0,\eta)\psi(0,\eta)\,d\eta\\
&{}&\geq\int_0^T\int_{\mathbb{H}^n}(\mathcal{K}\ast_{_{\mathbb{H}}}|u|^p)|u|^q\psi(t,\eta)\,d\eta\,dt+\int_0^T\int_{\mathbb{H}^n}u\,\Delta_{\mathbb{H}}\psi(t,\eta)\,d\eta\,dt\\
&{}&\quad+\int_0^T\int_{\mathbb{H}^n}u\,\psi_t(t,\eta)\,d\eta\,dt
\end{eqnarray*}
for all nonnegative compactly supported test function $\psi\in C^{1,2}_{t,x}([0,\infty)\times\mathbb{H}^n)$.\\
We choose
$$\psi(t,\eta):= \varphi^\ell(\eta) \varphi^\ell_3(t):=\varphi_1^\ell(x)\varphi_1^\ell(y)\varphi_2^\ell(\tau) \varphi^\ell_3(t),$$
 with
$$\varphi_1(x):=\Phi\left(\frac{|x|}{T^{\frac 12}}\right),\,\varphi_1(y):=\Phi\left(\frac{|y|}{T^{\frac 12}}\right),\,\varphi_2(\tau):=\Phi\left(\frac{|\tau|}{T}\right),\,\varphi_3(t):=\Phi\left(\frac{t}{T}\right),$$
where $\ell\gg 1$, and $\Phi$ is a smooth nonnegative non-increasing function such that
\[
\Phi(r)=\left\{\begin {array}{ll}
\displaystyle{1}&\displaystyle{\quad\text{if }0\leq r\leq 1/2,}\\\\
\displaystyle{\searrow}&\displaystyle{\quad\text{if }1/2\leq r\leq 1,}\\\\
\displaystyle{0}&\displaystyle{\quad\text {if }r\geq 1.}
\end {array}\right.
\]
Then
\begin{eqnarray}\label{3}
\int_0^T\int_{\mathcal{B}}(\mathcal{K}\ast_{_{\mathbb{H}}}|u|^p)|u|^q\psi(t,\eta)\,d\eta\,dt+\int_{\mathcal{B}}u_0(\eta)\varphi^\ell(\eta)\,d\eta&\leq&-\int_0^T\int_{\mathcal{C}}u\,\varphi^\ell_3(t)\,\Delta_{\mathbb{H}}\varphi^\ell(\eta)\,d\eta\,dt\nonumber\\
&{}&\quad -\int_{\frac{T}{2}}^T\int_{\mathcal{B}}u\,\varphi^\ell(\eta)\partial_t(\varphi^\ell_3(t))\,d\eta\,dt\nonumber\\
&=:&I_1+I_2,
\end{eqnarray}
where
$$\mathcal{B}=\{\eta=(x,y,\tau)\in\mathbb{H}^n;\,\,|x|^2,|y|^2,|\tau|\leq T\},$$ and $$\mathcal{C}=\{\eta=(x,y,\tau)\in\mathbb{H}^n;\,\,\frac{T^{\frac 12}}{2}\leq |x|,|y|\leq T^{\frac 12},\,\frac{T}{2}\leq|\tau|\leq T\}.$$
Let us start to estimate $I_1$. Using the following H\"older's inequality
$$\int ab\leq\left(\int a^{\frac{p+q}{2}}\right)^{\frac{2}{p+q}} \left(\int a^{\frac{p+q}{p+q-2}}\right)^{\frac{p+q-2}{p+q}} $$
we have
\begin{eqnarray}\label{4}
I_1&\leq&\int_0^T\int_{\mathcal{C}}|u|\,\varphi^\ell_3(t)\,\left|\Delta_{\mathbb{H}}\varphi^\ell(\eta)\right|\,d\eta\,dt\nonumber\\
&=&\int_0^T\int_{\mathcal{C}}|u|\,\psi^{\frac{2}{p+q}}(t,\eta)\psi^{-\frac{2}{p+q}}(t,\eta)\varphi^\ell_3(t)\,\left|\Delta_{\mathbb{H}}\varphi^\ell(\eta)\right|\,d\eta\,dt\nonumber\\
&\leq&\left(\int_0^T\int_{\mathcal{C}}|u|^{\frac{p+q}{2}}\psi(t,\eta)\,d\eta\,dt\right)^{\frac{2}{p+q}}\nonumber\\
&{}&\qquad\qquad\times\left(\int_0^T\int_{\mathcal{C}}\psi^{-\frac{2}{p+q-2}}(t,\eta)\varphi^{\frac{\ell(p+q)}{p+q-2}}_3(t)\,\left|\Delta_{\mathbb{H}}\varphi^\ell(\eta)\right|^{\frac{p+q}{p+q-2}}\,d\eta\,dt\right)^{\frac{p+q-2}{p+q}} .
\end{eqnarray}
Similarly,
\begin{eqnarray}\label{5}
I_2&\leq&\int_{\frac T2}^T\int_{\mathcal{B}}|u|\,\varphi^\ell(\eta)\left|\partial_t(\varphi^\ell_3(t))\right|\,d\eta\,dt\nonumber\\
&=&\int_{\frac T2}^T\int_{\mathcal{B}}|u|\,\psi^{\frac{2}{p+q}}(t,\eta)\psi^{-\frac{2}{p+q}}(t,\eta)\varphi^\ell(\eta)\left|\partial_t(\varphi^\ell_3(t))\right|\,d\eta\,dt\nonumber\\
&\leq&\left(\int_{\frac T2}^T\int_{\mathcal{B}}|u|^{\frac{p+q}{2}}\psi(t,\eta)\,d\eta\,dt\right)^{\frac{2}{p+q}}\nonumber\\
&{}&\qquad\qquad\times\left(\int_{\frac T2}^T\int_{\mathcal{B}}\psi^{-\frac{2}{p+q-2}}(t,\eta)\varphi^{\frac{\ell(p+q)}{p+q-2}}(\eta)\,\left|\partial_t(\varphi_3^\ell(t))\right|^{\frac{p+q}{p+q-2}}\,d\eta\,dt\right)^{\frac{p+q-2}{p+q}}.
\end{eqnarray}
Inserting \eqref{4}-\eqref{5} into \eqref{3}, we arrive at
\begin{eqnarray}\label{6}
&{}&\int_0^T\int_{\mathcal{B}}(\mathcal{K}\ast_{_{\mathbb{H}}}|u|^p)|u|^q\psi(t,\eta)\,d\eta\,dt+\int_{\mathcal{B}}u_0(\eta)\varphi^\ell(\eta)\,d\eta\nonumber\\
&{}&\leq\left(\int_0^T\int_{\mathcal{C}}|u|^{\frac{p+q}{2}}\psi(t,\eta)\,d\eta\,dt\right)^{\frac{2}{p+q}}\nonumber\\
&{}&\qquad\qquad\times\left(\int_0^T\int_{\mathcal{C}}\psi^{-\frac{2}{p+q-2}}(t,\eta)\varphi^{\frac{\ell(p+q)}{p+q-2}}_3(t)\,\left|\Delta_{\mathbb{H}}\varphi^\ell(\eta)\right|^{\frac{p+q}{p+q-2}}\,d\eta\,dt\right)^{\frac{p+q-2}{p+q}}\nonumber\\
&{}&\quad+\left(\int_{\frac T2}^T\int_{\mathcal{B}}|u|^{\frac{p+q}{2}}\psi(t,\eta)\,d\eta\,dt\right)^{\frac{2}{p+q}}\nonumber\\
&{}&\qquad\qquad\times\left(\int_{\frac T2}^T\int_{\mathcal{B}}\psi^{-\frac{2}{p+q-2}}(t,\eta)\varphi^{\frac{\ell(p+q)}{p+q-2}}(\eta)\,\left|\partial_t(\varphi_3^\ell(t))\right|^{\frac{p+q}{p+q-2}}\,d\eta\,dt\right)^{\frac{p+q-2}{p+q}}\nonumber\\
&{}&=J_1\,\left(\int_0^T\int_{\mathcal{C}}|u|^{\frac{p+q}{2}}\psi(t,\eta)\,d\eta\,dt\right)^{\frac{2}{p+q}}+\,J_2\,\left(\int_{\frac T2}^T\int_{\mathcal{B}}|u|^{\frac{p+q}{2}}\psi(t,\eta)\,d\eta\,dt\right)^{\frac{2}{p+q}}\qquad
\end{eqnarray}
where
$$J_1:=\left(\int_0^T\int_{\mathcal{C}}\psi^{-\frac{2}{p+q-2}}(t,\eta)\varphi^{\frac{\ell(p+q)}{p+q-2}}_3(t)\,\left|\Delta_{\mathbb{H}}\varphi^\ell(\eta)\right|^{\frac{p+q}{p+q-2}}\,d\eta\,dt\right)^{\frac{p+q-2}{p+q}},$$
and
$$J_2:=\left(\int_{\frac T2}^T\int_{\mathcal{B}}\psi^{-\frac{2}{p+q-2}}(t,\eta)\varphi^{\frac{\ell(p+q)}{p+q-2}}(\eta)\,\left|\partial_t(\varphi_3^\ell(t))\right|^{\frac{p+q}{p+q-2}}\,d\eta\,dt\right)^{\frac{p+q-2}{p+q}} .$$
Let us estimate $J_2$. As $\partial_t\varphi_3^\ell(t)=\ell\varphi_3^{\ell-1}(t)\partial_t\varphi_3(t)$, we have
\begin{eqnarray*}
J_2&\leq& C\,\left(\int_{\mathcal{B}}\varphi^\ell(\eta)\,d\eta\right)^{\frac{p+q-2}{p+q}}\left(\int_0^T\varphi_3^{\ell-\frac{p+q}{p+q-2}}(t)\,\left|\partial_t\varphi_3(t)\right|^{\frac{p+q}{p+q-2}}\,dt\right)^{\frac{p+q-2}{p+q}}\\
&=& C\,\left(\int_{\mathcal{B}}\varphi^\ell(\eta)\,d\eta\right)^{\frac{p+q-2}{p+q}}\left(\int_0^T\Phi^{\ell-\frac{p+q}{p+q-2}}\left(\frac{t}{T}\right)\,\left|\partial_t\Phi\left(\frac{t}{T}\right)\right|^{\frac{p+q}{p+q-2}}\,dt\right)^{\frac{p+q-2}{p+q}}.
\end{eqnarray*}
Letting
$$\widetilde{x}=\frac{x}{T^{\frac 12}},\qquad\widetilde{y}=\frac{y}{T^{\frac 12}},\qquad\widetilde{\tau}=\frac{\tau}{T},\qquad \widetilde{t}=\frac{t}{T},$$
and using the fact that $\varphi\leq1$ and meas$(\mathcal{B})= C\,T^{\frac Q2}$, we get
\begin{equation}\label{7}
J_2\leq C\,T^{\frac{Q(p+q-2)}{2(p+q)}-\frac{2}{p+q}}\left(\int_0^1\Phi^{\ell-\frac{p+q}{p+q-2}}(\tilde{t})\,\left|\Phi^{\prime}(\tilde{t})\right|^{\frac{p+q}{p+q-2}}\,d\widetilde{t}\right)^{\frac{p+q-2}{p+q}}\leq C\,T^{\frac{Q(p+q-2)}{2(p+q)}-\frac{2}{p+q}}.
\end{equation}
To estimate $J_1$, using \eqref{40}, we have
\begin{eqnarray*}
\left|\Delta_{\mathbb{H}}\varphi^\ell(\eta)\right|&=&\left|\Delta_{\mathbb{H}}\left(\varphi_1^\ell(x)\varphi_1^\ell(y)\varphi_2^\ell(\tau) \right)\right|\\
&\leq&\left|\Delta_x\varphi_1^\ell(x)\right|\varphi_1^\ell(y)\varphi_2^\ell(\tau)+\,\varphi_1^\ell(x)\left|\Delta_y\varphi_1^\ell(y)\right|\varphi_2^\ell(\tau)\\
&{}&+\,4(|x|^2+|y|^2)\varphi_1^\ell(x)\varphi_1^\ell(y)\left|\partial_\tau^2\varphi_2^\ell(\tau)\right| +\,4\sum_{j=1}^{n}|x_j|\varphi_1^\ell(x)\left|\partial_{y_j}\varphi_1^\ell(y)\right|\left|\partial_\tau\varphi_2^\ell(\tau)\right| \\
&{}&+\,4\sum_{j=1}^{n}|y_j|\varphi_1^\ell(y)\left|\partial_{x_j}\varphi_1^\ell(x)\right|\left|\partial_\tau\varphi_2^\ell(\tau)\right|,
\end{eqnarray*}
on $\mathcal{C}$. So
\begin{eqnarray*}
\left|\Delta_{\mathbb{H}}\varphi^\ell(\eta)\right|&\leq&\left[\ell(\ell-1)\varphi_1^{\ell-2}(x)|\nabla_x\varphi_1(x)|^2+\ell\varphi_1^{\ell-1}(x)|\Delta_x\varphi_1(x)|\right]\varphi_1^\ell(y)\varphi_2^\ell(\tau)\\
&{}&+\,\varphi_1^\ell(x)\left[\ell(\ell-1)\varphi_1^{\ell-2}(y)|\nabla_y\varphi_1(y)|^2+\ell\varphi_1^{\ell-1}(y)|\Delta_y\varphi_1(y)|\right]\varphi_2^\ell(\tau)\\
&{}&+\,4(|x|^2+|y|^2)\varphi_1^\ell(x)\varphi_1^\ell(y)\left[\ell(\ell-1)\varphi_2^{\ell-2}(\tau)|\partial_\tau\varphi_2(\tau)|^2\right] \\
&{}&+\,\ell\varphi_2^{\ell-1}(\tau)|\partial^2_\tau\varphi_2(\tau)|\\
&{}&+\,4\sum_{j=1}^{n}|x_j|\varphi_1^\ell(x)\left[\ell\varphi_1^{\ell-1}(y)\left|\partial_{y_j}\varphi_1(y)\right|\right]\left[\ell\varphi_2^{\ell-1}(\tau)\left|\partial_\tau\varphi_2(\tau)\right|\right] \\
&{}&+\,4\sum_{j=1}^{n}|y_j|\varphi_1^\ell(y)\left[\ell\varphi_1^{\ell-1}(x)\left|\partial_{x_j}\varphi_1(x)\right|\right]\left[\ell\varphi_2^{\ell-1}(\tau)\left|\partial_\tau\varphi_2(\tau)\right|\right],
\end{eqnarray*}
on $\mathcal{C}$. Substituting $\varphi_1$ and $\varphi_2$  we get
\begin{eqnarray*}
&{}&\left|\Delta_{\mathbb{H}}\varphi^\ell(\eta)\right|\\
&{}&\leq\left[\ell(\ell-1)\Phi^{\ell-2}\left(\frac{|x|}{T^{\frac 12}}\right)\left|\nabla_x\Phi\left(\frac{|x|}{T^{\frac 12}}\right)\right|^2\right.\\
&{}&\qquad\qquad\qquad\qquad\qquad\left.+\,\ell\Phi^{\ell-1}\left(\frac{|x|}{T^{\frac 12}}\right)\left|\Delta_x\Phi\left(\frac{|x|}{T^{\frac 12}}\right)\right|\right]\Phi^\ell\left(\frac{|y|}{T^{\frac 12}}\right)\Phi^\ell\left(\frac{|\tau|}{T}\right)\\
&{}&+\,\Phi^\ell\left(\frac{|x|}{T^{\frac 12}}\right)\left[\ell(\ell-1)\Phi^{\ell-2}\left(\frac{|y|}{T^{\frac 12}}\right)\left|\nabla_y\Phi\left(\frac{|y|}{T^{\frac 12}}\right)\right|^2\right.\\
&{}&\qquad\qquad\quad\qquad\qquad\qquad\qquad\quad\left. +\,\ell\Phi^{\ell-1}\left(\frac{|y|}{T^{\frac 12}}\right)\left|\Delta_y\Phi\left(\frac{|y|}{T^{\frac 12}}\right)\right|\right]\Phi^\ell\left(\frac{|\tau|}{T}\right)\\
&{}&+\,4(|x|^2+|y|^2)\Phi^\ell\left(\frac{|x|}{T^{\frac 12}}\right)\Phi^\ell\left(\frac{|y|}{T^{\frac 12}}\right)\left[\ell(\ell-1)\Phi^{\ell-2}\left(\frac{|\tau|}{T}\right)\left|\partial_\tau\Phi\left(\frac{|\tau|}{T}\right)\right|^2\right.\\
&{}&\qquad\qquad\qquad\quad\qquad\qquad\qquad\qquad\qquad\quad\left.+\ell\Phi^{\ell-1}\left(\frac{|\tau|}{T}\right)\left|\partial^2_\tau\Phi\left(\frac{|\tau|}{T}\right)\right|\right] \\
&{}&+\,4\sum_{j=1}^{n}|x_j|\Phi^\ell\left(\frac{|x|}{T^{\frac 12}}\right)\left[\ell\Phi^{\ell-1}\left(\frac{|y|}{T^{\frac 12}}\right)\left|\partial_{y_j}\Phi\left(\frac{|y|}{T^{\frac 12}}\right)\right|\right]\left[\ell\Phi^{\ell-1}\left(\frac{|\tau|}{T}\right)\left|\partial_\tau\Phi\left(\frac{|\tau|}{T}\right)\right|\right] \\
&{}&+\,4\sum_{j=1}^{n}|y_j|\Phi^\ell\left(\frac{|y|}{T^{\frac 12}}\right)\left[\ell\Phi^{\ell-1}\left(\frac{|x|}{T^{\frac 12}}\right)\left|\partial_{x_j}\Phi\left(\frac{|x|}{T^{\frac 12}}\right)\right|\right]\left[\ell\Phi^{\ell-1}\left(\frac{|\tau|}{T}\right)\left|\partial_\tau\Phi\left(\frac{|\tau|}{T}\right)\right|\right],
\end{eqnarray*}
on $\mathcal{C}$. By letting
$$\widetilde{x}=\frac{x}{T^{\frac 12}},\qquad\widetilde{y}=\frac{y}{T^{\frac 12}},\qquad\widetilde{\tau}=\frac{\tau}{T}.$$
we conclude that
\begin{align*}
&\left|\Delta_{\mathbb{H}}\varphi^\ell(\eta)\right|\\
&\leq\left[\ell(\ell-1)\Phi^{\ell-2}(|\widetilde{x}|)T^{-1} \left|\nabla_{\widetilde{x}}\Phi(|\widetilde{x}|)\right|^2+\ell\Phi^{\ell-1}(|\widetilde{x}|)T^{-1}\left|\Delta_{\widetilde{x}}\Phi(|\widetilde{x}|)\right|\right]\Phi^\ell(|\widetilde{y}|)\Phi^\ell(|\widetilde{\tau}|)\\
&+\,\Phi^\ell(|\widetilde{x}|)\left[\ell(\ell-1)\Phi^{\ell-2}(|\widetilde{y}|)T^{-1}\left|\nabla_{\widetilde{y}}\Phi(|\widetilde{y}|)\right|^2 +\ell\Phi^{\ell-1}(|\widetilde{y}|)T^{-1}\left|\Delta_{\widetilde{y}}\Phi(|\widetilde{y}|)\right|\right]\Phi^\ell(|\widetilde{\tau}|)\\
&+\,4\,T(|\widetilde{x}|^2+|\widetilde{y}|^2)\Phi^\ell(|\widetilde{x}|)\Phi^\ell(|\widetilde{y}|)\left[\ell(\ell-1)\Phi^{\ell-2}(|\widetilde{\tau}|){T^{-2}}\left|\partial_{\widetilde{\tau}}\Phi(|\widetilde{\tau}|)\right|^2\right.\\
&\qquad\qquad\qquad\qquad\qquad\qquad\qquad\qquad\qquad\qquad\qquad\left. +\ell\Phi^{\ell-1}(|\widetilde{\tau}|){T^{-2}}\left|\partial^2_{\widetilde{\tau}}\Phi(|\widetilde{\tau}|)\right|\right] \\
&+\,4\sum_{j=1}^{n}{T^{\frac 12}}|\widetilde{x}_j|\Phi^\ell(|\widetilde{x}|)\left[\ell\Phi^{\ell-1}(|\widetilde{y}|){T^{-\frac 12}} \left|\partial_{{\widetilde{y}}_j}\Phi(|\widetilde{y}|)\right|\right]\left[\ell\Phi^{\ell-1}(|\widetilde{\tau}|){T^{-1}}\left|\partial_{\widetilde{\tau}}\Phi(|\widetilde{\tau}|)\right|\right] \\
&+\,4\sum_{j=1}^{n}{T^{\frac 12}}|\widetilde{y}_j|\Phi^\ell(|\widetilde{y}|)\left[\ell\Phi^{\ell-1}(|\widetilde{x}|){T^{-\frac 12}}\left|\partial_{{\widetilde{x}}_j} \Phi(|\widetilde{x}|)\right|\right]\left[\ell\Phi^{\ell-1}(|\widetilde{\tau}|){T^{-1}}\left|\partial_{\widetilde{\tau}}\Phi(|\widetilde{\tau}|)\right|\right],
\end{align*}
on $\mathcal{C}$. Note that, as
$$\Phi\leq1\,\Rightarrow \Phi^{\ell}\leq \Phi^{\ell-1}\leq \Phi^{\ell-2},$$
we can easily see that
$$\left|\Delta_{\mathbb{H}}\varphi^\ell(\eta)\right|\leq\,C\,{T^{-1}}\left[\Phi^\ell(|\widetilde{x}|)\Phi^\ell(|\widetilde{y}|)\Phi^\ell(|\widetilde{\tau}|)\right]^{\ell-2},\qquad\hbox{for all}\,\,\eta\in\mathcal{C},$$
and therefore, using the fact that $\Phi,\varphi_3\leq1$, we conclude that
\begin{eqnarray}\label{8}
J_1&=&C\,\left(\int_0^T\varphi^\ell_3(t)\,dt\right)^{\frac{p+q-2}{p+q}}\left(\int_{\mathcal{C}}\varphi^{-\frac{2\ell}{p+q-2}}(\eta)\,\left|\Delta_{\mathbb{H}}\varphi^\ell(\eta)\right|^{\frac{p+q}{p+q-2}}\,d\eta\right)^{\frac{p+q-2}{p+q}}\nonumber\\
&\leq&C\,T^{-1}\left(\int_0^T\varphi^\ell_3(t)\,dt\right)^{\frac{p+q-2}{p+q}}\left(\int_{\mathcal{\widetilde{C}}} \left[\Phi(|\widetilde{x}|)\Phi(|\widetilde{y}|)\Phi(|\widetilde{\tau}|)\right]^{\ell-\frac{2(p+q)}{p+q-2}}T^{\frac Q2}\,d\widetilde{\eta}\right)^{\frac{p+q-2}{p+q}}\nonumber\\
&\leq&C\,T^{-1+\frac{p+q-2}{p+q}+\frac{Q(p+q-2)}{2(p+q)}}\nonumber\\
&\leq&C\,T^{-\frac{2}{p+q}+\frac{Q(p+q-2)}{2(p+q)}},
\end{eqnarray}
where we have used the fact that $\ell\gg1$.\\
Using \eqref{7}-\eqref{8}, we get from \eqref{6} that
\begin{eqnarray}\label{9}
\int_0^T\int_{\mathcal{B}}(\mathcal{K}\ast_{_{\mathbb{H}}}|u|^p)|u|^q\psi(t,\eta)\,d\eta\,dt+\int_{\mathcal{B}}u_0(\eta)\varphi^\ell(\eta)\,d\eta \nonumber\\ 
\leq C\,T^{\frac{Q}{2}-\frac{2+Q}{p+q}}\left(\textbf{I} +\textbf{J}\right),
\end{eqnarray}
where
\[
\textbf{I}:= \left(\int_0^T\int_{\mathcal{C}}|u|^{\frac{p+q}{2}}\psi(t,\eta)\,d\eta\,dt\right)^{\frac{2}{p+q}}
\]
and
\[
\textbf{J}:= \left(\int_{\frac T2}^T\int_{\mathcal{B}}|u|^{\frac{p+q}{2}}\psi(t,\eta)\,d\eta\,dt\right)^{\frac{2}{p+q}}.
\]
By the Cauchy-Schwarz inequality, we have
\begin{eqnarray*}
\left(\int_0^T\int_{\mathcal{C}}|u|^{\frac{p+q}{2}}\psi(t,\eta)\,d\eta\,dt\right)^{\frac{2}{p+q}}&=&\left(\int_0^T1\cdotp\int_{\mathcal{C}}|u|^{\frac{p+q}{2}}\psi(t,\eta)\,d\eta\,dt\right)^{\frac{2}{p+q}}\\
&\leq&T^{\frac{1}{p+q}}\left(\int_0^T\left(\int_{\mathcal{C}}|u|^{\frac{p+q}{2}}\psi(t,\eta)\,d\eta\right)^2\,dt\right)^{\frac{1}{p+q}}.
\end{eqnarray*}
Similarly,
$$
\left(\int_{\frac T2}^T\int_{\mathcal{B}}|u|^{\frac{p+q}{2}}\psi(t,\eta)\,d\eta\,dt\right)^{\frac{2}{p+q}}\leq T^{\frac{1}{p+q}}\left(\int_{\frac T2}^T\left(\int_{\mathcal{B}}|u|^{\frac{p+q}{2}}\psi(t,\eta)\,d\eta\right)^2\,dt\right)^{\frac{1}{p+q}}.
$$
Therefore,
\begin{eqnarray}\label{100}
&{}&\int_0^T\int_{\mathcal{B}}(\mathcal{K}\ast_{_{\mathbb{H}}}|u|^p)|u|^q\psi(t,\eta)\,d\eta\,dt+\int_{\mathcal{B}}u_0(\eta)\varphi^\ell(\eta)\,d\eta\nonumber\\
&{}&\leq C\,T^{\frac{Q}{2}-\frac{1+Q}{p+q}}\left[\left(\int_0^T\left(\int_{\mathcal{C}}|u|^{\frac{p+q}{2}}\psi(t,\eta)\,d\eta\right)^2\,dt\right)^{\frac{1}{p+q}}\right.\nonumber\\
&{}&\qquad\qquad\qquad\qquad\quad\left.+\,\left(\int_{\frac T2}^T\left(\int_{\mathcal{B}}|u|^{\frac{p+q}{2}}\psi(t,\eta)\,d\eta\right)^2\,dt\right)^{\frac{1}{p+q}}\right]
\end{eqnarray}
By letting
$$J(t):=\int_{\mathcal{B}}|u|^{\frac{p+q}{2}}\psi(t,\eta)\,d\eta=\int_{\mathbb{H}^n}|u|^{\frac{p+q}{2}}\psi(t,\eta)\,d\eta,\quad\hbox{for almost all}\,\,t\geq0,$$
we get
\begin{eqnarray}\label{10}
&{}&\int_0^T\int_{\mathcal{B}}(\mathcal{K}\ast_{_{\mathbb{H}}}|u|^p)|u|^q\psi(t,\eta)\,d\eta\,dt+\int_{\mathcal{B}}u_0(\eta)\varphi^\ell(\eta)\,d\eta\nonumber\\
&{}&\leq C\,T^{\frac{Q}{2}-\frac{1+Q}{p+q}}\left(\int_0^TJ^2(t)\,dt\right)^{\frac{1}{p+q}}.
\end{eqnarray}
To estimate the left-hand side of \eqref{10}, we have
$$(\mathcal{K}\ast_{_{\mathbb{H}}}|u|^p)(\eta)=\int_{\mathbb{H}^n}\mathcal{K}(d_{_{\mathbb{H}}}(\eta,\xi))|u(\xi)|^p\,d\xi\geq \int_{\mathcal{B}}\mathcal{K}(d_{_{\mathbb{H}}}(\eta,\xi))|u(\xi)|^p\,d\xi.$$
Note that $|\eta|_{_{\mathbb{H}}},\,|\xi|_{_{\mathbb{H}}}\leq \sqrt[4]{5}\,T^{\frac 12}\leq 2\,T^{\frac 12}$ on $\mathcal{B}$, therefore, using \cite{Yang}, we have
$$d_{_{\mathbb{H}}}(\eta,\xi)=|\xi^{-1}\circ \eta|_{_{\mathbb{H}}}\leq 3\left(|\xi|_{_{\mathbb{H}}}+|\eta|_{_{\mathbb{H}}}\right)\leq 12\,T^{\frac 12},\quad\hbox{on}\,\,\mathcal{B},$$
which implies that
$$\mathcal{K}(d_{_{\mathbb{H}}}(\eta,\xi))\geq \mathcal{K}(12\,T^{\frac 12}),\quad\hbox{on}\,\,\mathcal{B},$$
for all $T\gg1$, namely $12\,T^{\frac 12}>R_0$.
So,
$$(\mathcal{K}\ast_{_{\mathbb{H}}}|u|^p)(\eta)\geq  \mathcal{K}(12\,T^{\frac 12})\int_{\mathcal{B}}|u(\xi)|^p\,d\xi\quad\hbox{for all}\,\,\eta\in\mathcal{B}.$$
Then
\begin{eqnarray}\label{11}
&{}&\int_{\mathcal{B}}(\mathcal{K}\ast_{_{\mathbb{H}}}|u|^p)|u(\eta)|^q\psi(t,\eta)\,d\eta\nonumber\\
&{}&\geq  \mathcal{K}(12\,T^{\frac 12})\int_{\mathcal{B}}\int_{\mathcal{B}}|u(\xi)|^p|u(\eta)|^q\psi(t,\eta)\,d\xi\,d\eta\nonumber\\
&{}&\geq\mathcal{K}(12\,T^{\frac 12})\int_{\mathbb{H}^n}\int_{\mathbb{H}^n}|u(\xi)|^p\psi(t,\xi)|u(\eta)|^q\psi(t,\eta)\,d\xi\,d\eta,
\end{eqnarray}
where we have used that $\psi\leq 1$ and $\psi(\cdotp,t)\equiv 0$ outside of $\mathcal{B}$ for all $t\geq0$.
On the other hand, using Cauchy-Schwarz' inequality, we have
\begin{eqnarray}\label{12}
&{}&\int_{\mathbb{H}^n}\int_{\mathbb{H}^n}|u(\xi)|^{\frac{p+q}{2}}\psi(t,\xi)|u(\eta)|^{\frac{p+q}{2}}\psi(t,\eta)\,d\xi\,d\eta\nonumber\\
&{}&=\int_{\mathbb{H}^n}\int_{\mathbb{H}^n}|u(\xi)|^{\frac{p}{2}}\psi^{\frac 12}(t,\xi)|u(\eta)|^{\frac{q}{2}}\psi^{\frac 12}(t,\eta)|u(\xi)|^{\frac{q}{2}}\psi^{\frac 12}(t,\xi)|u(\eta)|^{\frac{p}{2}}\psi^{\frac 12}(t,\eta)\,d\xi\,d\eta\nonumber\\
&{}&\leq\left(\int_{\mathbb{H}^n}\int_{\mathbb{H}^n}|u(\xi)|^p\psi(t,\xi)|u(\eta)|^q\psi(t,\eta)\,d\xi\,d\eta\right)^{\frac 12}\nonumber\\
&{}&\qquad\qquad\qquad\qquad\qquad\times\left(\int_{\mathbb{H}^n}\int_{\mathbb{H}^n}|u(\xi)|^q\psi(t,\xi)|u(\eta)|^p\psi(t,\eta)\,d\xi\,d\eta\right)^{\frac 12}\nonumber\\
&{}&=\int_{\mathbb{H}^n}\int_{\mathbb{H}^n}|u(\xi)|^p\psi(t,\xi)|u(\eta)|^q\psi(t,\eta)\,d\xi\,d\eta.
\end{eqnarray}
By \eqref{11} and \eqref{12}, we conclude that
\begin{eqnarray}\label{13}
&{}&\int_{\mathcal{B}}(\mathcal{K}\ast_{_{\mathbb{H}}}|u|^p)|u(\eta)|^q\psi(t,\eta)\,d\eta\nonumber\\
&{}&\geq \mathcal{K}(12\,T^{\frac 12})\int_{\mathbb{H}^n}\int_{\mathbb{H}^n}|u(\xi)|^{\frac{p+q}{2}}\psi(t,\xi)|u(\eta)|^{\frac{p+q}{2}}\psi(t,\eta)\,d\xi\,d\eta\nonumber\\
&{}&= \mathcal{K}(12\,T^{\frac 12})\left(\int_{\mathbb{H}^n}|u(\eta)|^{\frac{p+q}{2}}\psi(t,\eta)\,d\eta\right)^2\nonumber\\
&{}&=\mathcal{K}(12\,T^{\frac 12})\,J^2(t).
\end{eqnarray}
Combining \eqref{10} and \eqref{13}, we infer that
\begin{equation}\label{14}
\mathcal{K}(12\,T^{\frac 12})\int_0^TJ^2(t)\,dt+\,\int_{\mathcal{B}}u_0(\eta)\varphi^\ell(\eta)\,d\eta\leq C\,T^{\frac{Q}{2}-\frac{1+Q}{p+q}}\left(\int_0^TJ^2(t)\,dt\right)^{\frac{1}{p+q}}.
\end{equation}
As $\displaystyle \int_{\mathbb{H}^n}u_0(\eta)\,d\eta>0\Longrightarrow \int_{\mathcal{B}}u_0(\eta)\varphi^\ell(\eta)\,d\eta\geq0$, we arrive at
$$
\mathcal{K}(12\,T^{\frac 12})\int_0^TJ^2(t)\,dt\leq C\,T^{\frac{Q}{2}-\frac{1+Q}{p+q}}\left(\int_0^TJ^2(t)\,dt\right)^{\frac{1}{p+q}},
$$
i.e.
\begin{equation}\label{15}
\left(\int_0^TJ^2(t)\,dt\right)^{\frac{p+q-1}{p+q}}\leq \frac{C}{\mathcal{K}(12\,T^{\frac 12})\,T^{\frac{1+Q}{p+q}-\frac{Q}{2}}},
\end{equation}
As
$$ \limsup_{R\rightarrow\infty}\mathcal{K}(R)R^{\frac{2Q+2}{p+q}-Q}>0,$$
there exists a sequence $\{R_j\}_j$ such that 
\begin{equation}\label{16}
R_j\rightarrow+\infty\qquad\hbox{and}\qquad \mathcal{K}(R_j)R_j^{\frac{2Q+2}{p+q}-Q}\longrightarrow\ell>0,\quad\hbox{as}\,\,j\rightarrow\infty.
\end{equation}
Without loss of generality, we may assume that $R_j>R_{j-1}$ for all $j>1$.\\

\noindent {\bf If $\ell=\infty$}, replacing $T$ by $R_j^2$, we have
$$
\left(\int_0^{R_j^2}J^2(t)\,dt\right)^{\frac{p+q-1}{p+q}}\leq \frac{C}{\mathcal{K}(12\,R_j)\,R_j^{\frac{2(1+Q)}{p+q}-Q}},
$$
passing to the limit when $j\rightarrow\infty$, using \eqref{16}, and the monotone convergence theorem, we infer that
$$
\int_0^{\infty}J^2(t)\,dt=0,
$$
which implies that $J\equiv 0$ a.e. on $\mathbb{R}^+$. By \eqref{14}, we conclude that
$$\int_{\mathcal{B}}u_0(\eta)\varphi^\ell(\eta)\,d\eta\leq 0,\qquad\hbox{for all}\,\,T\gg1.$$
Letting $T\rightarrow\infty$, using $u_0\in L^1(\mathbb{H}^n)$ and the dominated convergence theorem, we arrive at
$$0<\int_{\mathbb{H}^n}u_0(\eta)\,d\eta\leq 0,$$
contradiction.\\

\noindent {\bf If $\ell<\infty$}, then \eqref{15} shows that $J\in L^2(0,\infty)$. This implies that
$$\int_0^{R_j^2}\left(\int_{\mathcal{C}_j}|u|^{\frac{p+q}{2}}\psi(t,\eta)\,d\eta\right)^2\,dt,\quad \int_{\frac{R_j^2}{2}}^{R_j^2}\left(\int_{\mathcal{B}_j}|u|^{\frac{p+q}{2}}\psi(t,\eta)\,d\eta\right)^2\,dt\longrightarrow 0,$$
when $j\rightarrow\infty$, where
$$\mathcal{B}_j=\{\eta=(x,y,\tau)\in\mathbb{H}^n;\,\,|x|^2,|y|^2,|\tau|\leq R_j^2\},$$ and $$\mathcal{C}_j=\{\eta=(x,y,\tau)\in\mathbb{H}^n;\,\,\frac{R_j}{2}\leq |x|,|y|\leq R_j,\,\frac{R_j^2}{2}\leq|\tau|\leq R_j^2\}.$$
By \eqref{100} and \eqref{13}, and by replacing $T$ by $R_j^2$, we conclude that
$$
\left(\int_0^{R_j^2}J^2(t)\,dt\right)^{\frac{p+q-1}{p+q}}\leq \frac{C}{\mathcal{K}(12\,R_j)\,R_j^{\frac{2(1+Q)}{p+q}-Q}}o(1),\qquad\hbox{as}\,\,j\rightarrow\infty.
$$
Again, passing to the limit when $j\rightarrow\infty$, using \eqref{16} and $\ell\in(0,\infty)$, we get
$$
\int_0^{\infty}J^2(t)\,dt=0,
$$
which implies as above that
$$0<\int_{\mathbb{H}^n}u_0(\eta)\,d\eta\leq 0,$$
contradiction.\\

\noindent (ii) As $T>1$, we have
\begin{eqnarray*}
\int_{\mathcal{B}}u_0(\eta)\varphi^\ell(\eta)\,d\eta&\geq& \int_{\mathcal{C}_0}u_0(\eta)\varphi^\ell(\eta)\,d\eta\\
&=&\int_{\mathcal{C}_0}u_0(\eta)\,d\eta\\
&\geq& \varepsilon \int_{\mathcal{C}_0}(1+|\eta|_{_{\mathbb{H}}}^2)^{-\gamma/2}\,d\eta\\
&\geq& \varepsilon\,C \int_{\mathcal{C}_0}\left(\frac{T}{2}+ \frac{T}{2}\right)^{-\gamma/2}\,d\eta\\
&=&\varepsilon\,C T^{-\frac{\gamma}{2}}\,\hbox{meas}(\mathcal{C}_0)\\
&=&\varepsilon\,C T^{\frac{Q-\gamma}{2}},
\end{eqnarray*}
where 
\begin{equation}\label{C_0}
\mathcal{C}_0:=\{\eta=(x,y,\tau)\in\mathbb{H}^n;\,\, |x|,|y|\leq \frac{T^{\frac 12}}{2},\,|\tau|\leq \frac{T}{2}\}.
\end{equation} 
Therefore by repeating the same calculation as in the subcritical case (see \eqref{14}), we get
$$
\mathcal{K}(12\,T^{\frac 12})\int_0^TJ^2(t)\,dt+\,\varepsilon\,C T^{\frac{Q-\gamma}{2}}\leq C\,T^{\frac{Q}{2}-\frac{1+Q}{p+q}}\left(\int_0^TJ^2(t)\,dt\right)^{\frac{1}{p+q}},
$$
i.e.
\begin{eqnarray*}
\int_0^TJ^2(t)\,dt+\,\frac{\varepsilon\,C}{\mathcal{K}(12\,T^{\frac 12}) T^{\frac{\gamma-Q}{2}}}&\leq& \frac{C}{\mathcal{K}(12\,T^{\frac 12})\,T^{\frac{1+Q}{p+q}-\frac{Q}{2}}}\left(\int_0^TJ^2(t)\,dt\right)^{\frac{1}{p+q}}\\
&\leq&\frac{C}{\left(\mathcal{K}(12\,T^{\frac 12})\,T^{\frac{1+Q}{p+q}-\frac{Q}{2}}\right)^{\frac{p+q}{p+q-1}}}+\frac{1}{2}\int_0^TJ^2(t)\,dt,
\end{eqnarray*}
where we have used the following Young's inequality
$$ab\leq \frac{1}{2}a^{p+q}+Cb^{\frac{p+q}{p+q-1}}.$$
This implies that
$$
\frac{\varepsilon\,C}{\mathcal{K}(12\,T^{\frac 12})\, T^{\frac{\gamma-Q}{2}}}\leq\frac{C}{\left(\mathcal{K}(12\,T^{\frac 12})\,T^{\frac{1+Q}{p+q}-\frac{Q}{2}}\right)^{\frac{p+q}{p+q-1}}},
$$
i.e.
$$
\varepsilon\leq C\,\left[\mathcal{K}(12\,T^{\frac 12})\right]^{-\frac{1}{p+q-1}}\,T^{\frac{\gamma-Q}{2}-\frac{2(1+Q)-Q(p+q)}{2(p+q-1)}},
$$
and then
$$
\varepsilon\leq C\,\liminf_{T\longrightarrow\infty}\left(\left[\mathcal{K}(12\,T^{\frac 12})\right]^{-\frac{1}{p+q-1}}\,T^{\frac{\gamma-Q}{2}-\frac{2(1+Q)-Q(p+q)}{2(p+q-1)}}\right)=0,
$$
we get a contradiction.
\end{proof}


\subsection*{Acknowledgment}
Ahmad Fino is supported by the Research Group Unit, College of Engineering and Technology, American University of the Middle East.



\noindent{\bf Address}:\\
College of Engineering and Technology, American University of the Middle East, Kuwait.\\
\vspace{-7mm}
\begin{verbatim}
e-mail:  ahmad.fino@aum.edu.kw
\end{verbatim}
 Department of Mathematics, Faculty of Arts and Science, Khalifa University, P.O. Box: 127788,  Abu Dhabi, UAE.\\
\vspace{-7mm}
\begin{verbatim}
e-mail:  mokhtar.kirane@ku.ac.ae
\end{verbatim}
Department of Mathematics, Faculty of Sciences, Lebanese University, Tripoli, P.O. Box 1352, Lebanon
\vspace{-7mm}
\begin{verbatim}
e-mail:  bbarakeh@ul.edu.lb
\end{verbatim}
Department of Mathematics, Sultan Qaboos University,  FracDiff Research Group (DR/RG/03),  P.O. Box 36, Al-Khoud 123, Muscat, Oman
vspace{-7mm}
\begin{verbatim}
e-mail:  skerbal@squ.edu.om
\end{verbatim}

\end{document}